\documentclass[12pt]{article}

\usepackage{latexsym,amsmath,amscd,amssymb,graphics,float}
\usepackage{enumerate}

\usepackage{graphicx,lscape}

\usepackage[colorlinks]{hyperref}
\usepackage{url}

\begin{document}

\newenvironment{proof}[1][Proof]{\textbf{#1.} }{\ \rule{0.5em}{0.5em}}

\newtheorem{theorem}{Theorem}[section]
\newtheorem{definition}[theorem]{Definition}
\newtheorem{lemma}[theorem]{Lemma}
\newtheorem{remark}[theorem]{Remark}
\newtheorem{proposition}[theorem]{Proposition}
\newtheorem{corollary}[theorem]{Corollary}
\newtheorem{example}[theorem]{Example}

\numberwithin{equation}{section}
\newcommand{\ep}{\varepsilon}
\newcommand{\R}{{\mathbb  R}}
\newcommand\C{{\mathbb  C}}
\newcommand\Q{{\mathbb Q}}
\newcommand\Z{{\mathbb Z}}
\newcommand{\N}{{\mathbb N}}

\newcommand{\bfi}{\bfseries\itshape}

\newsavebox{\savepar}
\newenvironment{boxit}{\begin{lrbox}{\savepar}
\begin{minipage}[b]{15.5cm}}{\end{minipage}\end{lrbox}
\fbox{\usebox{\savepar}}}

\title{{\bf Asymptotic bp-stabilization of a given closed invariant set}}
\author{R\u{a}zvan M. Tudoran}

\date{}
\maketitle \makeatother

\begin{abstract}
Given a closed invariant set $\mathcal{C}$ of a dynamical system generated by a smooth vector field, $X$, for each $\lambda > 0$, we construct a control vector field, $X_{0}^{\lambda}$, such that the perturbed dynamics generated by the vector field $X+X_{0}^{\lambda}$, globally asymptotically bp-stabilizes the invariant set $\mathcal{C}$, that is, $\mathcal{C}$ attracts every bounded positive orbit of the perturbed dynamical system.
\end{abstract}

\medskip



\section{Introduction}
\label{section:one}

Asymptotic stabilization is an important technique specific to dynamical systems theory, with applicability in all sciences dealing with mathematical models described by dynamical systems, see e.g. \cite{astolfi}, \cite{byrnes}, \cite{chaturvedi}, and the references therein. Roughly speaking, this procedure is used in order to stabilize asymptotically a given invariant set of a dynamical system, and is usually achieved by means of a special type of perturbation, classically referred to as \textit{control}. Some of the main invariant sets to which the asymptotic stabilization procedure is addressed are certain types of orbits, such as equilibria and periodic orbits. Regarding some asymptotic stabilization results in this direction, see e.g. \cite{lin}, \cite{bbpt}, \cite{bc}, \cite{tudoranPO}, and the references therein. In the present work we deal with a special type of asymptotic stabilization, namely the so called \textit{global asymptotic bp-stabilization}. More exactly, a closed invariant set of a dynamical system is globally asymptotically bp-stable if it attracts every bounded positive orbit of the system.

The aim of this article is to provide an explicit method in order to globally asymptotically bp-stabilize a given closed invariant set of a dynamical system generated by a smooth vector field. More precisely, given a smooth vector field $X\in\mathfrak{X}(U)$ (defined on an open subset $U\subseteq\mathbb{R}^n$) and a closed invariant set $\mathcal{C}$ (defined as the preimage $\mathbf{D}^{-1}(\{\mathbf{d}\})$ corresponding to a value $\mathbf{d}\in \operatorname{Im}(\mathbf{D})$ of a smooth function $\mathbf{D}\in\mathcal{C}^{\infty}(U,\mathbb{R}^{p})$, where $1\leq p\leq n$), we construct a family of control vector fields $\left\{ X_{0}^{\lambda} \right\} _{\lambda > 0}\subset\mathfrak{X}(\operatorname{Mrk}(\mathbf{D}))$ (where $\operatorname{Mrk}(\mathbf{D})\subseteq U$ is the open set consisting of the maximal rank points of $\mathbf{D}$) such that the perturbed dynamics generated by the vector field $X+X_{0}^{\lambda}$, globally asymptotically bp-stabilizes the closed invariant set $\mathcal{C}$, i.e. $\mathcal{C}$ attracts every bounded positive orbit of the dynamical system $\dot{\mathbf{x}}=X(\mathbf{x})+X_{0}^{\lambda}(\mathbf{x})$, $\mathbf{x}\in\operatorname{Mrk}(\mathbf{D})$.

The structure of this article is the following. In the second section we prepare the settings of the problem by characterizing the set of dynamical systems (generated by smooth vector fields) with a prescribed closed invariant set. The third section contains the main result of this work which gives an explicit method to globally asymptotically bp-stabilize a given closed invariant set of a dynamical system generated by a smooth vector field defined on an open subset of $\mathbb{R}^n$. The last part of the article presents an example which illustrates the theoretical results obtained in this work.

\section{Dynamical systems with a prescribed closed invariant set}

Let us start this section by recalling a classical result which states that a nonempty closed subset $\mathcal{C}\subset \mathbb{R}^n$ (given as the preimage $\Sigma^{\mathbf{D},\mathbf{d}}:=\mathbf{D}^{-1}(\{\mathbf{d}\})$ corresponding to some value $\mathbf{d}:=(d_1,\dots,d_p)\in \operatorname{Im}(\mathbf{D})$ of a smooth function $\mathbf{D}:=(D_1,\dots,D_p):U\subseteq \mathbb{R}^n\rightarrow \mathbb{R}^p$, $1\leq p\leq n$, defined on an open subset $U\subseteq \mathbb{R}^n$) is invariant by the dynamics generated by a smooth vector field $X\in\mathfrak{X}(U)$, if there exist some smooth functions $h_{ij}\in\mathcal{C}^{\infty}(U,\mathbb{R})$, $i,j\in\{1,\dots,p\}$, such that 
\begin{align}\label{eqimp}
\mathcal{L}_{X}D_i = \sum_{j=1}^{p}h_{ij} (D_j - d_j), ~~~ \forall i\in\{1,\dots,p\},
\end{align}
where the notation $\mathcal{L}_{X}$ stands for the Lie derivative along the vector field $X$. Moreover, the above condition is also necessary if the rank of $\mathbf{D}$ is maximal (i.e., it equals to $p$) at every point of $\Sigma^{\mathbf{D},\mathbf{d}}$.

In order to give a geometric description of the set of smooth vector fields $X\in\mathfrak{X}(U)$ satisfying the condition \eqref{eqimp} (and hence keeping invariant $\Sigma^{\mathbf{D},\mathbf{d}}$), we shall need a result from \cite{tudoran} which provides a constructive method to characterize the class of smooth vector fields (defined on a smooth Riemannian manifold) which dissipate a given set of scalar quantities, with a-priori defined dissipation rates.

The following result from \cite{tudoran}, is the key ingredient to obtain the main results of this article.

\begin{theorem}[\cite{tudoran}]\label{MTD}
Let $(M,g)$ be an $n$-dimensional smooth Riemannian manifold, and fix $p\in \mathbb{N}\setminus \{0\}$ a nonzero natural number. Let $h_1, \dots, h_p \in \mathcal{C}^{\infty}(U,\mathbb{R})$ be a given set of smooth functions defined on an open subset $U\subseteq M$, and respectively let $D_1,\dots, D_p\in \mathcal{C}^{\infty}(U,\mathbb{R})$ be given, such that $\{\nabla D_1 , \dots,\nabla D_p\}\subset \mathfrak{X}(U)$ forms a set of pointwise linearly independent vector fields on $U$.

Then the set of solutions $X\in \mathfrak{X}(U)$ of the system
\begin{equation}\label{EQA}
\begin{split}
\left\{\begin{array}{l}
\mathcal{L}_{X}D_1 = h_1\\
\dots\\
\mathcal{L}_{X}D_p = h_p,\\
\end{array}\right.
\end{split}
\end{equation}
forms the affine distribution
$$
\mathfrak{A}[X_0;\nabla D_1 , \dots, \nabla D_p ]:=X_0 + \mathfrak{X}[\nabla D_1, \dots, \nabla D_p],
$$
locally generated by the set of vector fields
$$
\left\{X_0\right\}\biguplus \left\{\star\left( \bigwedge_{i=1, i\neq k}^{n-p} Z_i \wedge \bigwedge_{j=1}^{p} \nabla D_j \right): k\in\{1,\dots, n-p \}\right\}
$$
where
$$
X_{0}:=\left\| \bigwedge_{i=1}^{p} \nabla D_i \right\|_{p}^{-2}\cdot\sum_{i=1}^{p}(-1)^{n-i}h_i \Theta_i, ~~
\Theta_i := \star\left[ \bigwedge_{j=1, j\neq i}^{p} \nabla D_j \wedge \star\left(\bigwedge_{j=1}^{p} \nabla D_j \right)\right],
$$
and respectively the set of locally defined vector fields $\{\nabla D_1, \dots, \nabla D_p, Z_1, \dots, Z_{n-p}\}$ forms a moving frame. The notation ``$\star$" stands for the Hodge star operator on multivector fields, and $\nabla$ stands for the gradient operator associated to the Riemannian metric $g$. 
\end{theorem}
In contrast with the vector fields $\nabla D_1 , \dots,\nabla D_p$, which are globally defined on $U$, the vector fields $Z_1, \dots, Z_{n-p}$ exist in general only locally, around each point $x\in U$, in some open neighborhood $U_x \subseteq U$. \textbf{Nevertheless, the vector field $X_0$ is a globally defined solution of \eqref{EQA}.} 

Returning to the problem regarding the set of vector fields which verifies the condition \eqref{eqimp} (and hence keep invariant $\Sigma^{\mathbf{D},\mathbf{d}}$), we obtain the following result which is a direct application of Theorem \ref{MTD}. Before stating the result, let us denote by $\operatorname{Mrk}(\mathbf{D})\subseteq U$, the open subset of $U$ consisting of the maximal rank points of the smooth function $\mathbf{D}=(D_1,\dots,D_p):U\subseteq \mathbb{R}^n\rightarrow \mathbb{R}^p$, i.e. the points $\mathbf{x}\in U$ such that the vectors $\nabla D_1 (\mathbf{x}), \dots, \nabla D_p (\mathbf{x})$ are linearly independent. Recall that $\operatorname{Mrk}(\mathbf{D})$ is an open subset of $U$ contained in the set of regular points of $D$, where a point $\mathbf{x}_0 \in U$ is a regular point of $\mathbf{D}$ if there exists an open neighborhood $U_{\mathbf{x}_0}\subseteq U$ such that $ \operatorname{rank}(\mathrm{d}\mathbf{D} (\mathbf{x}))=\operatorname{rank}(\mathrm{d}\mathbf{D} (\mathbf{x}_0))$, for all $\mathbf{x}\in U_{\mathbf{x}_0}$. Recall also that the set of regular points of $\mathbf{D}$ is an open dense subset of $U$ in contrast with $\operatorname{Mrk}(\mathbf{D})$ which is open but not necessarily dense. The rank of $\mathrm{d}\mathbf{D} (\cdot)$ is constant on each connected component of the set of regular points of $\mathbf{D}$.

\begin{theorem}
Let $\mathcal{C}\subset \mathbb{R}^n$ be a nonempty closed subset given as the preimage $\Sigma^{\mathbf{D},\mathbf{d}}:=\mathbf{D}^{-1}(\{\mathbf{d}\})$ corresponding to some value $\mathbf{d}:=(d_1,\dots,d_p)\in \operatorname{Im}(\mathbf{D})$ of a smooth function $\mathbf{D}:=(D_1,\dots,D_p):U\subseteq \mathbb{R}^n\rightarrow \mathbb{R}^p$, $1\leq p\leq n$, defined on an open subset $U\subseteq \mathbb{R}^n$. Let $h_{ij}\in\mathcal{C}^{\infty}(U,\mathbb{R})$, $i,j\in\{1,\dots,p\}$, be a given set of smooth functions. 

If one denotes $h_i :=\sum_{j=1}^{p}h_{ij} (D_j - d_j), i\in\{1,\dots,p\}$, then the set of solutions $X\in \mathfrak{X}(\operatorname{Mrk}(\mathbf{D}))$ of the system
\begin{equation}\label{EQAA}
\begin{split}
\left\{\begin{array}{l}
\mathcal{L}_{X}D_1 = h_1\\
\dots\\
\mathcal{L}_{X}D_p = h_p,\\
\end{array}\right.
\end{split}
\end{equation}
forms the affine distribution
$$
\mathfrak{A}[X_0;\nabla D_1 , \dots, \nabla D_p ] =X_0 + \mathfrak{X}[\nabla D_1, \dots, \nabla D_p],
$$
locally generated by the set of vector fields
$$
\left\{X_0\right\}\biguplus \left\{\star\left( \bigwedge_{i=1, i\neq k}^{n-p} Z_i \wedge \bigwedge_{j=1}^{p} \nabla D_j \right): k\in\{1,\dots, n-p \}\right\}
$$
where
$$
X_{0} =\left\| \bigwedge_{i=1}^{p} \nabla D_i \right\|_{p}^{-2}\cdot\sum_{i=1}^{p}(-1)^{n-i}h_i \Theta_i, ~~~~~
\Theta_i  = \star\left[ \bigwedge_{j=1, j\neq i}^{p} \nabla D_j \wedge \star\left(\bigwedge_{j=1}^{p} \nabla D_j \right)\right],
$$
and respectively the set of locally defined vector fields $$\{\nabla D_1, \dots, \nabla D_p, Z_1, \dots, Z_{n-p}\}$$ forms a moving frame.
\end{theorem}
\begin{proof}
The proof follows directly from the Theorem \ref{MTD}.
\end{proof}

Note that each vector field from the affine distribution $\mathfrak{A}[X_0;\nabla D_1 , \dots, \nabla D_p ]$ keeps dynamically invariant the set $\Sigma^{\mathbf{D},\mathbf{d}}\cap \operatorname{Mrk}(\mathbf{D})$. Moreover, in the case when $\mathbf{d}$ is a regular value of the smooth function $\mathbf{D}=(D_1,\dots,D_p)$, then $\Sigma^{\mathbf{D},\mathbf{d}}\subset \operatorname{Mrk}(\mathbf{D})$, and hence in this case each vector field from the affine distribution $\mathfrak{A}[X_0;\nabla D_1 , \dots, \nabla D_p ]$ keeps dynamically invariant the set $\Sigma^{\mathbf{D},\mathbf{d}}$; in this case, $\Sigma^{\mathbf{D},\mathbf{d}}$ is a smooth submanifold of $\mathbb{R}^n$ of dimension $n-p$. Recall that by Sard's theorem, almost all points in the image of the smooth function $\mathbf{D}$ are regular values, i.e. the set of singular values of $\mathbf{D}$ is a set of Lebesgue measure zero in $\mathbb{R}^p$.

\section{Global asymptotic bp-stabilization of a given closed invariant set}

This section contains the main result of this work which gives an explicit method to globally asymptotically bp-stabilize a given closed invariant set of a dynamical system generated by a smooth vector field defined on an open subset of $\mathbb{R}^n$.

Before stating the main result of this section, we recall some terminology we shall need in the sequel. In order to do that, we consider a smooth vector field $X\in\mathfrak{X}(U)$ defined on an open set $U\subseteq\mathbb{R}^n$. Then for each $\overline{\mathbf{x}}\in U$ we denote by $t\in I_{\overline{\mathbf{x}}}\subseteq \mathbb{R}\mapsto \mathbf{x}(t;\overline{\mathbf{x}})\in U$, the solution of the Cauchy problem ${\mathrm{d}\mathbf{x}}/{\mathrm{d}t}=X(\mathbf{x}(t))$, $\mathbf{x}(0)=\overline{\mathbf{x}}$, defined on the maximal domain $I_{\overline{\mathbf{x}}}\subseteq \mathbb{R}$, where $I_{\overline{\mathbf{x}}}$ is an open interval of $\mathbb{R}$ containing the origin. If the \textit{positive orbit} of $\overline{\mathbf{x}}\in U$ (i.e. the set $\{\mathbf{x}(t;\overline{\mathbf{x}}):t\geq 0\}$) is contained in a compact subset of $U$, then the solution $x(t;\overline{\mathbf{x}})$ is defined for all $t\in[0,\infty)$. Denoting by $\omega(\overline{\mathbf{x}}):=\{\mathbf{y}\in U: \exists (t_{n})_{n\in\mathbb{N}}\subset [0,\infty),~t_{n} < t_{n+1},~{t_n}\rightarrow {\infty}~\text{s.t.}~\lim_{n\rightarrow \infty}\mathbf{x}(t_n;\overline{\mathbf{x}})=\mathbf{y}\}$, the $\omega-$\textit{limit set} of $\overline{\mathbf{x}}$, we get that for each $\overline{\mathbf{x}}\in U$ with bounded positive orbit, $\omega(\overline{\mathbf{x}})$ is a nonempty, invariant, compact and connected subset of $U$. Moreover, $\mathbf{x}(t;\overline{\mathbf{x}})\rightarrow\omega(\overline{\mathbf{x}})$ as $t\rightarrow \infty$, i.e. for every $\varepsilon>0$ there exists $T>0$ such that $\operatorname{dist}(\mathbf{x}(t;\overline{\mathbf{x}}),\omega(\overline{\mathbf{x}}))<\varepsilon$, for all $t>T$. For details see e.g., \cite{hartman}, \cite{verhulst}. 

\begin{definition}[\cite{tudoranDef}]
Let $X\in\mathfrak{X}(U)$ be a smooth vector field defined on an open set $U\subseteq\mathbb{R}^n$. A closed and invariant subset $\mathcal{A}\subset U$ is called \textbf{globally  bp-attracting set} of the dynamical system generated by $X$, if for \textbf{every} point $\overline{\mathbf{x}}\in U$ such that the set $\{\mathbf{x}(t;\overline{\mathbf{x}}): t\geq 0\}$ is bounded, the integral curve of $X$ starting from $\overline{\mathbf{x}}$ approaches $\mathcal{A}$ as $t\rightarrow \infty$.
\end{definition}
Recall from \cite{tudoranDef} that in contrast to global attractors, a globally bp-attracting set is not required to be compact, and also needs not be connected, its connectivity being related to the connectivity of the open set $U$. 

Let us state now the main result of this article.
\begin{theorem}\label{mainTHM}
Let $\dot {\mathbf{x}} = X(\mathbf{x}), ~ \mathbf{x}\in U$, be a dynamical system generated by a smooth vector field, $X\in\mathfrak{X}(U)$, defined on an open set $U\subseteq\mathbb{R}^{n}$. Assume there exists a set of smooth functions, $h_{ij}\in\mathcal{C}^{\infty}(U,\mathbb{R})$, $i,j\in\{1,\dots,p\}$, such that
\begin{equation}\label{EQAU}
\begin{split}
\left\{\begin{array}{l}
\mathcal{L}_{X}D_1 = h_1\\
\dots\\
\mathcal{L}_{X}D_p = h_p,\\
\end{array}\right.
\end{split}
\end{equation}
where $\mathbf{D}:=(D_1,\dots,D_p):U\subseteq \mathbb{R}^n\rightarrow \mathbb{R}^p$, $1\leq p\leq n$, is a smooth function, $\mathbf{d}:=(d_1,\dots,d_p)\in \operatorname{Im}(\mathbf{D})$, is a fixed point in the image of $\mathbf{D}$, and $h_i :=\sum_{j=1}^{p}h_{ij} (D_j - d_j), i\in\{1,\dots,p\}$.

Then, for each strictly positive real number, $\lambda >0$, there exists a smooth vector field, $X_{0}^{\lambda}\in\mathfrak{X}(\operatorname{Mrk}(\mathbf{D}))$, given by 
$$
X_{0}^{\lambda}:=\left\| \bigwedge_{i=1}^{p} \nabla D_i \right\|_{p}^{-2}\cdot\sum_{i=1}^{p}(-1)^{n-i+1}[h_i +\lambda(D_i - d_i)]\Theta_i, ~~~~~
\Theta_i := \star\left[ \bigwedge_{j=1, j\neq i}^{p} \nabla D_j \wedge \star\left(\bigwedge_{j=1}^{p} \nabla D_j \right)\right],
$$
such that the perturbed dynamical system,
\begin{equation}\label{sp}
\dot {\mathbf{x}} = X(\mathbf{x}) + X_{0}^{\lambda}(\mathbf{x}), ~ \mathbf{x}\in\operatorname{Mrk}(\mathbf{D}),
\end{equation}
satisfies the following statements. 

\begin{itemize}
\item [(a)] The set $\Sigma^{\mathbf{D},\mathbf{d}}:=\mathbf{D}^{-1}(\{\mathbf{d}\})$ is a dynamically invariant set of both, the dynamical system $\dot {\mathbf{x}} = X(\mathbf{x}), ~ \mathbf{x}\in U$, and the perturbed system $\dot {\mathbf{x}} = X(\mathbf{x})+ X_{0}^{\lambda}(\mathbf{x}), ~ \mathbf{x}\in \operatorname{Mrk}(\mathbf{D})$.
\item [(b)] The set $\Sigma^{\mathbf{D},\mathbf{d}}\cap \operatorname{Mrk}(\mathbf{D})$ is a globally bp-attracting set of the perturbed system \eqref{sp}. More precisely, for every $\overline{\mathbf{x}} \in\operatorname{Mrk}(\mathbf{D})$, such that the set $\{\mathbf{x}(t;\overline{\mathbf{x}}): t\geq 0\}$ is bounded, $\mathbf{x}(t;\overline{\mathbf{x}})\rightarrow \Sigma^{\mathbf{D},\mathbf{d}}\cap \operatorname{Mrk}(\mathbf{D})$ as $t\rightarrow \infty$.
\item [(c)] Assume there exists $K\subset \operatorname{Mrk}(\mathbf{D})$ a positively invariant compact set of the system \eqref{sp}. Then for every $\overline{\mathbf{x}}\in K$ we have that $\mathbf{x}(t;\overline{\mathbf{x}})\rightarrow \Sigma^{\mathbf{D},\mathbf{d}}\cap \operatorname{Mrk}(\mathbf{D})$ as $t\rightarrow \infty$.
\item [(d)] If $\Sigma^{\mathbf{D},\mathbf{d}}\cap \operatorname{Mrk}(\mathbf{D})$ contains isolated points, then each such a point $\overline{\mathbf{x}}\in \Sigma^{\mathbf{D},\mathbf{d}}\cap\operatorname{Mrk}(\mathbf{D})$ is an asymptotically stable equilibrium point of the perturbed system \eqref{sp}.
\item [(e)] If $\mathbf{d}\in \operatorname{Im}(\mathbf{D})$ is a regular value of $\mathbf{D}$, then $\Sigma^{\mathbf{D},\mathbf{d}}$ is a globally bp-attracting set of the perturbed system \eqref{sp}. More precisely, for every $\overline{\mathbf{x}} \in\operatorname{Mrk}(\mathbf{D})$, such that the set $\{\mathbf{x}(t;\overline{\mathbf{x}}): t\geq 0\}$ is bounded, $\mathbf{x}(t;\overline{\mathbf{x}})\rightarrow \Sigma^{\mathbf{D},\mathbf{d}}$ as $t\rightarrow \infty$.
\item [(f)] Assume there exists $K\subset \operatorname{Mrk}(\mathbf{D})$ a positively invariant compact set of the system \eqref{sp}. If $\mathbf{d}\in \operatorname{Im}(\mathbf{D})$ is a regular value of $\mathbf{D}$, then for every $\overline{\mathbf{x}}\in K$ we have that $\mathbf{x}(t;\overline{\mathbf{x}})\rightarrow \Sigma^{\mathbf{D},\mathbf{d}}$ as $t\rightarrow \infty$.
\item [(g)] If $\mathbf{d}\in \operatorname{Im}(\mathbf{D})$ is a regular value of $\mathbf{D}$, and if $\Sigma^{\mathbf{D},\mathbf{d}}$ contains isolated points, then each such a point $\overline{\mathbf{x}}\in \Sigma^{\mathbf{D},\mathbf{d}}$ is an asymptotically stable equilibrium point of the perturbed system \eqref{sp}.
\end{itemize}
\end{theorem}
\begin{proof}
We start by defining the smooth function $F:\operatorname{Mrk}(\mathbf{D})\rightarrow [0,\infty)$ given by
$$
F(\mathbf{x}):=(D_1 (\mathbf{x})-d_1)^2 +\dots + (D_p (\mathbf{x})-d_p)^2, ~ \forall \mathbf{x}\in\operatorname{Mrk}(\mathbf{D}).
$$
Since the vector field $X_{0}^{\lambda}\in\mathfrak{X}(\operatorname{Mrk}(\mathbf{D}))$ verifies by construction the equalities
\begin{equation}\label{rei}
\mathcal{L}_{X_{0}^{\lambda}}D_i = - [h_i+\lambda (D_i - d_i)], ~ \forall i\in\{1,\dots,p\},
\end{equation}
we get that
\begin{align*}
\begin{split}
\mathcal{L}_{X_{0}^{\lambda}}F &= \sum_{i=1}^{p}\mathcal{L}_{X_{0}^{\lambda}}(D_i - d_i)^2 = \sum_{i=1}^{p}2(D_i - d_i)\mathcal{L}_{X_{0}^{\lambda}}(D_i - d_i)=\sum_{i=1}^{p}2(D_i - d_i)\mathcal{L}_{X_{0}^{\lambda}}D_i\\
&= \sum_{i=1}^{p}2(D_i - d_i)[-h_i-\lambda (D_i - d_i)]= \sum_{i=1}^{p}(-2 h_i)(D_i - d_i) + (-2\lambda)\sum_{i=1}^{p}(D_i - d_i)^2 \\
&=\sum_{i=1}^{p}(-2 h_i)(D_i - d_i) + (-2\lambda) F.
\end{split}
\end{align*}
Consequently, as the vector field $X$ verifies by definition the equalities
\begin{equation*}
\mathcal{L}_{X}D_i = h_i, ~ \forall i\in\{1,\dots,p\},
\end{equation*}
using a computation similar to the one given above, we get that
$$
\mathcal{L}_{X}F =\sum_{i=1}^{p}2 h_i(D_i - d_i),
$$
and hence we obtain the following relations
\begin{align}\label{lder}
\begin{split}
\mathcal{L}_{X+X_{0}^{\lambda}}F &= \mathcal{L}_{X}F + \mathcal{L}_{X_{0}^{\lambda}}F\\
&=\sum_{i=1}^{p}2 h_i(D_i - d_i) + \left(\sum_{i=1}^{p}(-2 h_i)(D_i - d_i) + (-2\lambda) F \right)\\
&=(-2\lambda) F.
\end{split}
\end{align}
Let $\overline{\mathbf{x}}\in \operatorname{Mrk}(\mathbf{D})$ be given, and let $t\in I_{\overline{\mathbf{x}}}\subseteq\mathbb{R}\mapsto \mathbf{x}(t;\overline{\mathbf{x}})\in\operatorname{Mrk}(\mathbf{D})$ be the integral curve of the vector field $X+X_{0}^{\lambda}\in\mathfrak{X}(\operatorname{Mrk}(\mathbf{D}))$ such that $\mathbf{x}(0;\overline{\mathbf{x}})=\overline{\mathbf{x}}$, where $I_{\overline{\mathbf{x}}}\subseteq\mathbb{R}$ stands for the maximal domain of definition of the solution $\mathbf{x}(\cdot;\overline{\mathbf{x}})$.

Using the relation \eqref{lder}, we obtain that
\begin{equation*}
\dfrac{\mathrm{d}}{\mathrm{d}t}F(\mathbf{x}(t;\overline{\mathbf{x}}))= (-2\lambda) F(\mathbf{x}(t;\overline{\mathbf{x}})), ~ \forall t\in I_{\overline{\mathbf{x}}},
\end{equation*}
and hence 
\begin{equation}\label{grw}
F(\mathbf{x}(t;\overline{\mathbf{x}}))= \exp(-2\lambda t) \cdot F(\overline{\mathbf{x}}), ~ \forall t\in I_{\overline{\mathbf{x}}}.
\end{equation}
Moreover, since the set of zeros of $F$ coincides with $\Sigma^{\mathbf{D},\mathbf{d}}\cap \operatorname{Mrk}(\mathbf{D})$, the following sets equality holds true:
\begin{equation}\label{zeroset}
\{\mathbf{x}\in\operatorname{Mrk}(\mathbf{D}) : (\mathcal{L}_{X+ X_{0}^{\lambda}}F)(\mathbf{x})=0\}=\Sigma^{\mathbf{D},\mathbf{d}}\cap \operatorname{Mrk}(\mathbf{D}).
\end{equation}

Let us prove now the first item of the conclusion.
\begin{itemize}
\item [(a)] Using the relations \eqref{EQAU}, it follows directly that $\Sigma^{\mathbf{D},\mathbf{d}}$ is an invariant set of the dynamical system $\dot {\mathbf{x}} = X(\mathbf{x}), ~ \mathbf{x}\in U$. 

In order to prove that $\Sigma^{\mathbf{D},\mathbf{d}}$ is an invariant set of the perturbed system, $\dot {\mathbf{x}} = X(\mathbf{x})+ X_{0}^{\lambda}(\mathbf{x}), ~ \mathbf{x}\in \operatorname{Mrk}(\mathbf{D})$, recall from \eqref{rei} that
\begin{align*}
\mathcal{L}_{X_{0}^{\lambda}}D_i &= - [h_i+\lambda (D_i - d_i)], ~ \forall i\in\{1,\dots,p\},
\end{align*}
and thus, using the relations \eqref{EQAU} we obtain
\begin{align*}
\mathcal{L}_{X+X_{0}^{\lambda}}D_i &= \mathcal{L}_{X}D_i + \mathcal{L}_{X_{0}^{\lambda}}D_i= h_i - [h_i+\lambda (D_i - d_i)]\\
&=-\lambda (D_i - d_i), ~ \forall i\in\{1,\dots,p\}.
\end{align*}

\item [(b)] Let $\overline{\mathbf{x}}\in\operatorname{Mrk}(\mathbf{D})$ be an arbitrary point of the open set $\operatorname{Mrk}(\mathbf{D})$ such that the set $\{\mathbf{x}(t;\overline{\mathbf{x}}): t\geq 0\}$ is bounded. We show that the $\omega-$limit set $\omega(\overline{\mathbf{x}})$ is a subset of  $\Sigma^{\mathbf{D},\mathbf{d}}\cap \operatorname{Mrk}(\mathbf{D})$. In order to do that, let $\mathbf{y}\in\omega(\overline{\mathbf{x}})$ be arbitrary chosen. As the set $\{\mathbf{x}(t;\overline{\mathbf{x}}): t\geq 0\}$ is bounded, we get that $[0,\infty)\subset I_{\overline{\mathbf{x}}}$, and hence there exists an increasing sequence $(t_{n})_{n\in\mathbb{N}}\subset [0,\infty)$, with $\lim_{n\rightarrow \infty}t_n =\infty$, such that $\lim_{n\rightarrow \infty}\mathbf{x}(t_n;\overline{\mathbf{x}})=\mathbf{y}$. Using the relation \eqref{grw}, it follows that
\begin{equation}\label{grwok}
F(\mathbf{x}(t;\overline{\mathbf{x}}))= \exp(-2\lambda t) \cdot F(\overline{\mathbf{x}}), ~ \forall t\in [0,\infty).
\end{equation}
For $t=t_n \geq 0$, $n\in\mathbb{N}$, the equality \eqref{grwok} becomes
\begin{equation*}
F(\mathbf{x}(t_n;\overline{\mathbf{x}}))= \exp(-2\lambda t_n) \cdot F(\overline{\mathbf{x}}), ~ \forall n\in \mathbb{N}.
\end{equation*}
Since $\lambda >0$, $\lim_{n\rightarrow \infty} t_n =\infty $, $\lim_{n\rightarrow \infty}\mathbf{x}(t_n;\overline{\mathbf{x}})=\mathbf{y}$, and $F$ is continuous, we get that $F(\mathbf{y})=0$. Taking into account that $F^{-1}(\{0\})=\Sigma^{\mathbf{D},\mathbf{d}}\cap \operatorname{Mrk}(\mathbf{D})$, it follows that $\mathbf{y}\in \Sigma^{\mathbf{D},\mathbf{d}}\cap \operatorname{Mrk}(\mathbf{D})$. As $\mathbf{y}\in\omega(\overline{\mathbf{x}})$ was arbitrary chosen, we obtain that $\omega(\overline{\mathbf{x}})\subseteq \Sigma^{\mathbf{D},\mathbf{d}}\cap \operatorname{Mrk}(\mathbf{D})$. Since $\mathbf{x}(t;\overline{\mathbf{x}})\rightarrow \omega(\overline{\mathbf{x}})\subseteq \Sigma^{\mathbf{D},\mathbf{d}}\cap \operatorname{Mrk}(\mathbf{D})$ as $t\rightarrow \infty$, it follows that $\mathbf{x}(t;\overline{\mathbf{x}})\rightarrow \Sigma^{\mathbf{D},\mathbf{d}}\cap \operatorname{Mrk}(\mathbf{D})$ as $t\rightarrow \infty$.
\item [(c)] The proof follows directly from item $(b)$ taking into account that for every $\overline{\mathbf{x}} \in K$, the set $\{\mathbf{x}(t;\overline{\mathbf{x}}): t\geq 0\}$ is bounded.
\item [(d)] Let $\mathbf{x}_e$ be an isolated point of $\Sigma^{\mathbf{D},\mathbf{d}}\cap\operatorname{Mrk}(\mathbf{D})$. Hence, there exists $U_{\mathbf{x}_e} \subseteq \operatorname{Mrk}(\mathbf{D})$, an open neighborhood of $\mathbf{x}_e$ such that $U_{\mathbf{x}_e}\cap(\Sigma^{\mathbf{D},\mathbf{d}}\cap\operatorname{Mrk}(\mathbf{D}))=\{\mathbf{x}_e \}$. 

Let $F: U_{\mathbf{x}_e}\rightarrow [0,\infty)$ be given by
$$
F(\mathbf{x}):=(D_1 (\mathbf{x})-d_1)^2 +\dots + (D_p (\mathbf{x})-d_p)^2, ~ \forall \mathbf{x}\in U_{\mathbf{x}_e}.
$$
Note that $\mathbf{x}_e$ is the unique solution of the equation $F(\mathbf{x})=0$ in $U_{\mathbf{x}_e}$. This follows directly taking into account that the set of zeros of $F$ located in $U_{\mathbf{x}_e}$ is  $U_{\mathbf{x}_e}\cap(\Sigma^{\mathbf{D},\mathbf{d}}\cap \operatorname{Mrk}(\mathbf{D}))$, and $U_{\mathbf{x}_e}\cap(\Sigma^{\mathbf{D},\mathbf{d}}\cap \operatorname{Mrk}(\mathbf{D}))=\{\mathbf{x}_e\}$.

Let us recall now from \eqref{lder} that the smooth function $F$ satisfies the relation
\begin{equation}
(\mathcal{L}_{X+X_{0}^{\lambda}}F) (\mathbf{x})= (-2\lambda)F(\mathbf{x}), ~ \forall \mathbf{x}\in U_{\mathbf{x}_e}.
\end{equation}
Hence, we get that $F(\mathbf{x}_e)=0$, $F(\mathbf{x})>0$, $(\mathcal{L}_{X+X_{0}^{\lambda}}F) (\mathbf{x})<0$, for every $\mathbf{x}\in U_{\mathbf{x}_e}\setminus\{\mathbf{x}_e\}$. Thus $F$ is a strict Lyapunov function associated to $\mathbf{x}_e$ and  consequently $\mathbf{x}_e$ is an asymptotically stable equilibrium point of the dynamical system \eqref{sp}.
\item [(e,f,g)]Each of the items $(e),(f),(g)$ follows directly from the corresponding one taking into account that if $\mathbf{d}\in \operatorname{Im}(\mathbf{D})$ is a regular value of $\mathbf{D}$ then $\Sigma^{\mathbf{D},\mathbf{d}}\subset \operatorname{Mrk}(\mathbf{D})$, and hence, in this case, $\Sigma^{\mathbf{D},\mathbf{d}}\cap\operatorname{Mrk}(\mathbf{D})=\Sigma^{\mathbf{D},\mathbf{d}}$.
\end{itemize}
\end{proof}

Next, we present an example which illustrates the main results obtained in Theorem \ref{mainTHM}.
\begin{example}
Let us consider the following dynamical system on $\mathbb{R}^2$
\begin{equation}\label{exem}
\left\{ \begin{array}{l}
 \dot x = x(x^2+y^2-1) \\
 \dot y = x^2+y^2-1,
 \end{array} \right.
\end{equation}
generated by the vector field $X:=x(x^2+y^2-1)\dfrac{\partial}{\partial x} + (x^2+y^2-1)\dfrac{\partial}{\partial y}\in\mathfrak{X}(\mathbb{R}^2)$.

The sets, $S_1:=\{(x,y)\in\mathbb{R}^2 : x=0\}$ and $S_2:=\{(x,y)\in\mathbb{R}^2 : x^2 +y^2 =1\}$, are both dynamically invariant sets of the system \eqref{exem}. This follows directly from \eqref{eqimp}, taking into account that
\begin{equation*}
\mathcal{L}_{X}(x)= x(x^2+y^2-1), ~ \mathcal{L}_{X}(x^2 +y^2)= 2(x^2+y)(x^2+y^2-1), ~\forall (x,y)\in\mathbb{R}^2.
\end{equation*}

Now we apply the Theorem \ref{mainTHM} in order to globally asymptotically bp-stabilize each of the invariant sets $S_1$, $S_2$, and $S_1\cap S_2$. 

\begin{enumerate}
\item [(i)] In order to globally asymptotically bp-stabilize $S_1$, following the notations from Theorem \ref{mainTHM}, we have that $S_1 = \Sigma^{\mathbf{D},\mathbf{d}}=\mathbf{D}^{-1}(\{\mathbf{d}\})$, where $\mathbf{D}=D_1 :\mathbb{R}^{2}\rightarrow \mathbb{R}$, $\mathbf{D}(x,y)=D_1(x,y)=x$, $\forall(x,y)\in\mathbb{R}^2$, and $\mathbf{d}=d_1 =0$. Thus, $\operatorname{Mrk}(\mathbf{D})=\mathbb{R}^2$, and hence $\Sigma^{\mathbf{D},\mathbf{d}}\cap \operatorname{Mrk}(\mathbf{D})=S_1$.

As $h_1=\mathcal{L}_{X} D_1=x(x^2+y^2-1)$, and $\Theta_1 =\star[\star(\nabla D_1)]=-\dfrac{\partial}{\partial x}$, it follows that the control vector field, $X_{0}^{\lambda}\in\mathfrak{X}(\operatorname{Mrk}(\mathbf{D}))=\mathfrak{X}(\mathbb{R}^2)$, $\lambda >0$, is given by
\begin{align*}
X_{0}^{\lambda}&=\|\nabla D_1\|^{-2}(-1)^{2-1+1}[h_1+\lambda (D_1-d_1)]\Theta_1=\left[-x(x^2 + y^2-1)-\lambda x\right]\dfrac{\partial}{\partial x}.
\end{align*}
Consequently, from Theorem \ref{mainTHM} we obtain that the perturbed system 
\begin{equation*}
\left\{ \begin{array}{l}
\dot x = -\lambda x \\
\dot y = x^2+y^2-1,
\end{array} \right.
\end{equation*}
generated by the vector field $X+X_{0}^{\lambda}=-\lambda x\dfrac{\partial}{\partial x}+(x^2+y^2-1)\dfrac{\partial}{\partial y}\in\mathfrak{X}(\mathbb{R}^2)$, globally asymptotically bp-stabilize the invariant set $S_1$.

\item [(ii)] In order to globally asymptotically bp-stabilize $S_2$, following the notations from Theorem \ref{mainTHM}, we have that $S_2 = \Sigma^{\mathbf{D},\mathbf{d}}=\mathbf{D}^{-1}(\{\mathbf{d}\})$, where $\mathbf{D}=D_1 :\mathbb{R}^{2}\rightarrow \mathbb{R}$, $\mathbf{D}(x,y)=D_1(x,y)=x^2+y^2$, $\forall(x,y)\in\mathbb{R}^2$, and $\mathbf{d}=d_1 =1$. Thus, $\operatorname{Mrk}(\mathbf{D})=\mathbb{R}^2 \setminus\{(0,0)\}$, and hence $\Sigma^{\mathbf{D},\mathbf{d}}\cap \operatorname{Mrk}(\mathbf{D})=S_2$.

As $h_1=\mathcal{L}_{X} D_1=2(x^2+y)(x^2+y^2-1)$, and $\Theta_1 =\star[\star(\nabla D_1)]=-2x\dfrac{\partial}{\partial x}-2y\dfrac{\partial}{\partial x}$, it follows that the control vector field, $X_{0}^{\lambda}\in\mathfrak{X}(\operatorname{Mrk}(\mathbf{D}))=\mathfrak{X}(\mathbb{R}^2 \setminus\{(0,0)\})$, $\lambda >0$, is given by
\begin{align*}
X_{0}^{\lambda}&=\|\nabla D_1\|^{-2}(-1)^{2-1+1}[h_1+\lambda (D_1-d_1)]\Theta_1\\
&=-\dfrac{x(x^2+y^2-1)(2x^2+2y+\lambda)}{2(x^2+y^2)}\dfrac{\partial}{\partial x}-\dfrac{y(x^2+y^2-1)(2x^2+2y+\lambda)}{2(x^2+y^2)}\dfrac{\partial}{\partial y}.
\end{align*}
Consequently, from Theorem \ref{mainTHM} we obtain that the perturbed system 
\begin{equation*}
\left\{ \begin{array}{l}
\dot x = \dfrac{x(x^2+y^2-1)(2y^2-2y-\lambda)}{2(x^2+y^2)}\\
\dot y = \dfrac{(x^2+y^2-1)(2x^2-2x^2 y-\lambda y)}{2(x^2+y^2)},
\end{array} \right.
\end{equation*}
generated by the vector field $X+X_{0}^{\lambda}\in\mathfrak{X}(\mathbb{R}^2 \setminus\{(0,0)\})$
$$X+X_{0}^{\lambda}=\dfrac{x(x^2+y^2-1)(2y^2-2y-\lambda)}{2(x^2+y^2)}\dfrac{\partial}{\partial x}+\dfrac{(x^2+y^2-1)(2x^2-2x^2 y-\lambda y)}{2(x^2+y^2)}\dfrac{\partial}{\partial y},$$ globally asymptotically bp-stabilize the invariant set $S_2$.

\item [(iii)] In order to globally asymptotically bp-stabilize $S_1\cap S_2$, following the notations from Theorem \ref{mainTHM}, we have that $S_1\cap S_2 = \Sigma^{\mathbf{D},\mathbf{d}}=\mathbf{D}^{-1}(\{\mathbf{d}\})$, where $\mathbf{D}=(D_1,D_2) :\mathbb{R}^{2}\rightarrow \mathbb{R}^2$, $\mathbf{D}(x,y)=(D_1(x,y),D_2(x,y))=(x,x^2+y^2)$, $\forall(x,y)\in\mathbb{R}^2$, and $\mathbf{d}=(d_1,d_2)=(0,1)$. Thus, $\operatorname{Mrk}(\mathbf{D})=\mathbb{R}^2 \setminus\{(x,y):y=0\}=\mathbb{R}^2 \setminus Ox$, and hence $\Sigma^{\mathbf{D},\mathbf{d}}\cap \operatorname{Mrk}(\mathbf{D})=S_1\cap S_2=\{(0,-1),(0,1)\}$.

As $h_1=\mathcal{L}_{X} D_1=x(x^2+y^2-1)$, $h_2=\mathcal{L}_{X} D_2=2(x^2+y)(x^2+y^2-1)$, $\Theta_1 =\star[\nabla D_2\wedge\star(\nabla D_1\wedge\nabla D_2)]=-4y^2\dfrac{\partial}{\partial x}+4xy\dfrac{\partial}{\partial y}$, and $\Theta_2 =\star[\nabla D_1\wedge\star(\nabla D_1 \wedge\nabla D_2)]=2y\dfrac{\partial}{\partial y}$, it follows that the control vector field, $X_{0}^{\lambda}\in\mathfrak{X}(\operatorname{Mrk}(\mathbf{D}))=\mathfrak{X}(\mathbb{R}^2 \setminus Ox)$, $\lambda >0$, is given by
\begin{align*}
X_{0}^{\lambda}&=\|\nabla D_1 \wedge \nabla D_2\|^{-2}\left\{[h_1+\lambda (D_1-d_1)]\Theta_1 - [h_2+\lambda (D_2-d_2)]\Theta_2\right\}\\
&=-x(x^2+y^2-1+\lambda)\dfrac{\partial}{\partial x}+\left[\dfrac{\lambda (x^2-y^2+1)}{2y}-x^2-y^2+1\right]\dfrac{\partial}{\partial y}.
\end{align*}
Consequently, from Theorem \ref{mainTHM} we obtain that the perturbed system 
\begin{equation}\label{petdi}
\left\{ \begin{array}{l}
\dot x = -\lambda x\\
\dot y = \dfrac{\lambda (x^2-y^2+1)}{2y},
\end{array} \right.
\end{equation}
generated by the vector field $X+X_{0}^{\lambda}\in\mathfrak{X}(\mathbb{R}^2 \setminus Ox)$
$$X+X_{0}^{\lambda}=-\lambda x \dfrac{\partial}{\partial x}+\dfrac{\lambda (x^2-y^2+1)}{2y}\dfrac{\partial}{\partial y},$$ globally asymptotically bp-stabilize the invariant set $S_1\cap S_2 =\{(0,-1),(0,1)\}$. Moreover, as $(0,-1)$ and $(0,1)$ are isolated points of the invariant set $S_1\cap S_2$, it follows that each of them is an asymptotically stable equilibrium state of the perturbed dynamics \eqref{petdi}.

\end{enumerate} 
\end{example}


\bigskip
\bigskip

\noindent {\sc R.M. Tudoran}\\
West University of Timi\c soara\\
Faculty of Mathematics and Computer Science\\
Department of Mathematics\\
Blvd. Vasile P\^arvan, No. 4\\
300223 - Timi\c soara, Rom\^ania.\\
E-mail: {\sf razvan.tudoran@e-uvt.ro}\\
\medskip

\end{document}